\documentclass[a4paper]{article}

\usepackage[utf8]{inputenc}
\usepackage{lmodern}
\usepackage{array}
\usepackage{graphicx}
\usepackage{amssymb, amsfonts, amsmath, amsthm}
\usepackage[english]{babel}
\usepackage[pdfborder={0 0 0}]{hyperref}
\usepackage{tikz}
\usetikzlibrary{decorations}
\usetikzlibrary{decorations.pathreplacing}
\tikzset{individu/.style={draw,thick}}

\usepackage{subfigure}

\theoremstyle{plain}
\newtheorem{theorem}{Theorem}[section]

\newtheorem{lemma}[theorem]{Lemma}
\newtheorem{proposition}[theorem]{Proposition}

\theoremstyle{definition}

\theoremstyle{remark}
\newtheorem{remark}[theorem]{Remark}

\newcommand{\N}{\mathbb{N}}
\newcommand{\Z}{\mathbb{Z}}
\newcommand{\R}{\mathbb{R}}

\newcommand{\ind}[1]{\mathbf{1}_{\left\{#1\right\}}}

\newcommand{\floor}[1]{{\left\lfloor #1 \right\rfloor}}
\newcommand{\ceil}[1]{{\left\lceil #1 \right\rceil}}

\renewcommand{\rho}{\varrho}

\numberwithin{equation}{section}

\DeclareMathOperator{\E}{\mathbb{E}}

\renewcommand{\P}{\mathbb{P}}

\newcommand{\calF}{\mathcal{F}}

\newcommand{\calL}{\mathcal{L}}

\newcommand{\calG}{\mathcal{G}}

\renewcommand{\bar}[1]{\overline{#1}}

\title{Necessary and sufficient conditions for the convergence of the consistent maximal displacement of the branching random walk}
\author{Bastien Mallein\footnote{LAGA, Université Paris 13 (Villetaneuse).}}
\date{\today}

\newcommand{\T}{\mathbf{T}}
\renewcommand{\tilde}[1]{\widetilde{#1}}
\renewcommand{\hat}[1]{\widehat{#1}}

\begin{document}

\maketitle
\begin{abstract}
Consider a supercritical branching random walk on the real line. The consistent maximal displacement is the smallest of the distances between the trajectories followed by individuals at the $n$th generation and the boundary of the process. Fang and Zeitouni, and Faraud, Hu and Shi proved that under some integrability conditions, the consistent maximal displacement grows almost surely at rate $\lambda^* n^{1/3}$ for some explicit constant~$\lambda^*$. We obtain here a necessary and sufficient condition for this asymptotic behaviour to hold.
\end{abstract}

\section{Introduction}
\label{sec:introduction}

A branching random walk on $\R$ is a process defined as follows. It starts with one individual located at $0$ at time $0$. Its children are positioned in $\R$ according to a point process of law $\calL$, and form the first generation of the process. Then for any $n \in \N$, each individual in the $n$-th generation makes children around its current position according to an independent point process with law $\calL$. We write $\T$ for the genealogical tree of the population. For any $u \in \T$ we denote by $V(u)$ the position of the individual $u$ and by $|u|$ the generation to which $u$ belongs. The random marked tree $(\T,V)$ is the branching random walk with reproduction law $\calL$. We assume the Galton-Watson tree $\T$ is supercritical, i.e.
\begin{equation}
  \label{eqn:supercritical}
  \E\left( \sum_{|u|=1} 1 \right) > 1,
\end{equation}
and we write $S = \{\# \T = \infty\}$ for the survival event, which happens with positive probability under assumption \eqref{eqn:supercritical}. We also assume the branching random walk $(\T,V)$ is in the boundary case (in the sense of \cite{BiK05}):
\begin{equation}
  \label{eqn:boundary}
  \E\left( \sum_{|u|=1} e^{-V(u)} \right) = 1 \quad \mathrm{and} \quad \E\left( \sum_{|u|=1} V(u) e^{-V(u)} \right) = 0.
\end{equation}
Under these assumptions, Biggins \cite{Big76} proved that $\frac{1}{n}\min_{|u|=n} V(u)$ converges to $0$ almost surely on $S$. Any branching random walk with mild integrability assumption can be normalized to be in the boundary case, see e.g. Bérard and Gouéré \cite{BeG11}. We also assume that
\begin{equation}
  \label{eqn:variance}
  \sigma^2 := \E\left( \sum_{|u|=1} V(u)^2 e^{-V(u)} \right) < \infty.
\end{equation}

Let $n \in \N$. For any $u \in \T$ such that $|u|=n$ and $k \leq n$ we denote by $u_k$ the ancestor of $u$ alive at generation $k$. The consistent maximal displacement of the branching random walk is the quantity defined as
\begin{equation*}
  L_n := \min_{|u|=n} \max_{k \leq n} V(u_k).
\end{equation*}
It corresponds loosely to the maximal distance between the lower boundary of the branching random walk and the trajectory of the individual that stayed as close as possible to it. The asymptotic behaviour of $L_n$ has been studied by Fang and Zeitouni \cite{FaZ10} and by Fauraud, Hu and Shi \cite{FHS12}. Under stronger integrability assumptions, they proved that $L_n$ behaves as $\lambda^* n^{1/3}$ almost surely for some explicit $\lambda^* > 0$. The main result of this article is a necessary and sufficient condition for this asymptotic behaviour to hold. Roberts \cite{Rob14} computed the second order of the asymptotic behaviour of $L_n$ for a similar model, the branching Brownian motion.

This $O(n^{1/3})$ asymptotic behaviour for the consistent maximal displacement is non-obvious, as it is different of the asymptotic behaviour of the minimum of the branching random walk, $M_n := \min_{|u|=n} V(u)$. Indeed, it was proved by Addario-Berry and Reed \cite{ABR09} and by Hu and Shi \cite{HuS09} that under some additional assumptions, the minimal displacement satisfies
\[
  \lim_{n \to \infty} \frac{M_n}{\log n} = \frac{3}{2} \quad \text{in probability.}
\]
Thus, the consistent maximal displacement grows much faster than the minimal displacement does. Note that it was proven in \cite{Che14b} that the trajectory yielding to the minimal position at time $n$, when rescaled in time by $n$ and in space by $n^{1/2}$, converges toward a Brownian excursion. Therefore, the maximal distance from 0 of this trajectory is of order $n^{1/2}$, much larger than what is expected for the consistent maximal displacement. As a result, one conclude that particles realizing the minimal displacement and particles realizing the consistent maximal displacement form distinct sets.

We introduce the integrability assumption
\begin{equation}
  \label{eqn:integrability}
  \lim_{x \to \infty} x^2 \E\left( \sum_{|u|=1} e^{-V(u)} \ind{\log \left(\sum_{|v|=1} e^{-V(v)}\right) \geq x} \right) = 0.
\end{equation}
Observe that this assumption is weaker than the classical integrability assumptions \cite[Equation (1.4)]{Aid13} for branching random walks. These stronger conditions are necessary and sufficient to obtain the regularity of the asymptotic behaviour of many quantities associated to the extremal particles in the branching random walk, such as the minimal displacement $M_n$ \cite{Aid13}, or the derivative martingale \cite{Che14}.
\begin{theorem}
\label{thm:main}
We assume that \eqref{eqn:supercritical}, \eqref{eqn:boundary} and \eqref{eqn:variance} hold. Then \eqref{eqn:integrability} is a necessary and sufficient condition for
\[\lim_{n \to \infty} \frac{L_n}{n^{1/3}} = \left( \frac{3\pi^2 \sigma^2}{2} \right)^{1/3} \quad \text{a.s. on $S$}.\]
\end{theorem}

In the rest of the article, we denote by $\N$ the set of positive integers, $\Z_+$ the set of non-negative integers, $\R_+$ the set of non-negative real numbers. For any $x \in \R_+$, we write $[x] = [0,x] \cap \Z_+$ the set of non-negative integers that are smaller or equal to $x$.

The rest of the article is organised as follows. In Section \ref{sec:rw} we introduce the spinal decomposition of the branching random walk and the Mogul'ski\u\i{}'s small deviations estimate. These results are used to obtain a law of large numbers for $L_n$ in Section~\ref{sec:computation}, which is then used to obtain its almost sure asymptotic behaviour.

\section{Preliminary results}
\label{sec:rw}

In this section, we first introduce the spinal decomposition, that links additive moments of the branching random walk with random walk estimates. Using this result, to study the consistent maximal displacement, good estimates of the probability for random walks to stay in a small interval will be needed. We introduce them in a second time, by expanding Mogul'ski\u\i{} results \cite{Mog74} on small deviations for random walk trajectories.

\subsection{Spinal decomposition of the branching random walk}
\label{subsec:spinal}

For $n \in \Z_+$, we write $W_n = \sum_{|u|=n} e^{-V(u)}$ and $\calF_n = \sigma\left( V(u), |u| \in [n] \right)$. Under assumption \eqref{eqn:boundary}, $(W_n)$ is a non-negative $(\calF_n)$-martingale. We introduce the probability $\bar{\P}$ such that for any $n \in \Z_+$, $\bar{\P}_{|\calF_n} = W_n \cdot \P_{|\calF_n}$.

The spinal decomposition consists in an alternative description of $\bar{\P}$ as a branching random walk with a distinguished individual with a different reproduction law. It generalizes a similar construction for Galton-Watson processes, that can be found in \cite{LPP95}. This result has been proved by Lyons in \cite{Lyo97}. Let $\T$ be a tree, a spine of $\T$ is a sequence $w=(w_n) \in \T^{\Z_+}$ such that for any $n \in \Z_+$ and $k \in [n]$ we have $|w_n|=n$ and $(w_n)_k = w_k$. We write $\hat{\calL}$ for the law of the point process $(V(u), |u|=1)$ under the law $\bar{\P}$.

We now define the law $\hat{\P}$ of a branching random walk with spine $(\T,V,w)$. It starts with a unique individual $w_0$ located at $0$ at time $0$. Its children are positioned according to a point process of law $\hat{\calL}$. The individual $w_1$ is then chosen at random among these children $u$ with probability proportional to $e^{-V(u)}$. Similarly at each generation $n$, every individual $u$ makes children independently, according to law $\calL$ if $u \neq w_n$, or $\hat{\calL}$ otherwise; then $w_{n+1}$ is chosen at random among the children $v$ of $w_n$ with probability proportional to~$e^{-V(v)}$.
\begin{proposition}[Spinal decomposition, Lyons \cite{Lyo97}]
\label{pro:spinaldecomposition}
Assuming \eqref{eqn:boundary} and \eqref{eqn:variance}, for any $n \in \Z_+$, we have $\hat{\P}_{|\calF_n} = \bar{\P}_{|\calF_n}$, for any $|u|=n$,
\[
  \hat{\P}\left( w_n = u \left| \calF_n \right. \right) = e^{-V(u)}/W_n,
\]
and $(V(w_n), n \in \Z_+)$ is a centred random walk with variance $\sigma^2$.
\end{proposition}

The spinal decomposition is widely used in branching random walk literature. In particular, it implies the so-called many-to-one lemma, introduced for the first time by Kahane and Peyrière \cite{Pey,KaP}. For any $n \in \N$, for any measurable non-negative function $f : \R^{n+1} \to \R$, we have
\begin{align*}
  \E\left( \sum_{|u|=n} f(V(u_j), j \leq n) \right) &= \bar{\E}\left( \frac{1}{W_n} \sum_{|u|=n} f(V(u_j), j \leq n) \right)\\
  &= \hat{\E}\left( \sum_{|u|=n} \frac{e^{-V(u)}}{W_n} e^{V(u)} \sum_{|u|=n} f(V(u_j), j \leq n) \right)\\
  &= \hat{\E}\left( e^{V(w_n)} f(V(w_j), j \leq n) \right).
\end{align*}
In other words, to compute the mean of an additive functional of the branching random walk, it is enough to compute the mean of a similar functional for the sole random walk $(V(w_n), n \geq 0)$, up to an exponential tilting.

\begin{remark}
Note that \eqref{eqn:integrability} can be rewritten, using the spinal decomposition, as
\begin{equation}
  \label{eqn:integrabilitySpine}
  \lim_{x \to \infty} x^2 \hat{\P}\left(\sum_{|v|=1} e^{-V(v)} \geq e^x \right) = 0.
\end{equation}
In other words, this assumption translates into a condition on the tail of the distribution of the progeny of the spine particle.
\end{remark}

\subsection{Small deviations estimate for enriched random walk}
\label{subsec:mogulskii}

The spinal decomposition for the branching random walk allows to simplify branching random walk computations by focusing only on the spine particle. For example, to compute the average number of particles satisfying a given property, it is often enough to consider the probability for the spine to satisfy this property, under the size-biased law. In the rest of the article, we often have to deal with the fact that the progeny of the spine particle is usually correlated with its displacement. As a result, we develop in this section tight estimates on the small deviations for \emph{enriched random walks}, that we use as toy-models for the spine of a branching random walk.

An enriched random walk can be constructed as follows. Let $(X_n,\xi_n)_{n \in \N}$ be a sequence of i.i.d. vectors in $\R^2$ such that 
\begin{equation}
  \label{eqn:integrabilityrw}
  \E(X_n)=0 \quad \text{ and } \quad \E(X_n^2) = \sigma^2 \in (0,\infty),
\end{equation}
We denote by $S_n = S_0 + X_1 + \cdots + X_n$. For any $z \in \R$, $\P_z$ is the probability such that $\P_z(S_0= z)= 1$. We simply write $\P$ for $\P_0$. Setting $\xi_0=0$, an enriched random walk is the process $((S_n, \xi_n), n \in \Z_+)$, that takes values in $\R^2$.

Enriched random walks appear under other names in the literature. For example, the process $(S_{n-1} + \xi_n, n \geq 0)$ is often called a perturbed random walk. This process has been studied in the perpetuity literature (see e.g. \cite{AlI,ArG}).

We study in this section the probability for an enriched random walk to stay during $n$ units of time in an interval of width $o(n^{1/2})$, generalizing the Mogul'ski\u\i{} small deviations estimate \cite{Mog74}.

\begin{theorem}
\label{thm:newSpinalMogulskii}
Let $(a_n) \in \R_+^\N$ such that $\lim_{n \to \infty} a_n = \infty$ and $a_n =o(n^{1/2})$. We assume \eqref{eqn:integrabilityrw} and that $\lim_{x \to \infty} x^2 \P(\xi_1 > x) = \rho \in [0,\infty]$. For any continuous functions $f<g$ such that $f(0)<0<g(0)$, for all $\lambda>0$ and $t > 0$,
\begin{align*}
  &\lim_{n \to \infty} \frac{a_n^2}{n} \log \sup_{z \in [f(0),g(0)]} \P_{za_n}\left( \tfrac{S_j}{a_n} \in [f(j/n),g(j/n)], \tfrac{\xi_j}{a_n} \leq \lambda, j \in [t n] \right)\\
   =  &\lim_{n \to \infty} \frac{a_n^2}{n} \log \P\left(\tfrac{S_j}{a_n} \in [f(j/n),g(j/n)], \tfrac{\xi_j}{a_n} \leq \lambda, j \in [tn] \right)\\
  =  &- \int_0^t  \frac{\pi^2 \sigma^2}{2(g(s)-f(s))^2} ds - \frac{\rho t}{\lambda^2}.
\end{align*}
\end{theorem}

The rest of the section is devoted to the proof of Theorem~\ref{thm:newSpinalMogulskii}, using the same techniques as in \cite[Lemma 2.6]{Mal15b}. The first step relies on the following observation.

\begin{lemma}
\label{lem:cvFonction}
For $n \in \N$ and $t \in \R_+$, we introduce the functions $S^{(n)}_t = \frac{S_{\floor{nt}}}{n^{1/2}}$ and $ P^{(n)}_t = \sum_{j=1}^{\floor{nt}} \ind{\xi_j > n^{1/2}}$. We assume \eqref{eqn:integrabilityrw}, and that there exists an increasing sequence $(n_k) \in \N^\N$ and $\rho \in [0,\infty)$ such that
\begin{equation}
  \label{eqn:sequentialLimitXi}
  \lim_{k \to \infty} n_k \P\left(\xi_1 > (n_k)^{1/2}\right) = \rho.
\end{equation}
Then $\lim_{k \to \infty} (S^{(n_k)},P^{(n_k)}) = (\sigma B,P)$ in law, for the Skorokhod topology on the space of c\`adl\`ag functions, where $B$ is a Brownian motion and $P$ an independent Poisson process with parameter $\rho$.
\end{lemma}

This lemma shows that the events $\{ \tfrac{S_j}{a_n} \in [f(j/n),g(j/n)], j \in [n]\}$ and $\{ \xi_j \leq \lambda a_n, j \in [n]\}$ become asymptotically independent, which is an heuristic explanation for Theorem \ref{thm:newSpinalMogulskii}.

\begin{proof}
To lighten the notation, in this proof, every asymptotic behaviour written as $n \to \infty$ is implicitly considered along the subsequence $(n_k)$ that is given as a hypothesis.

We observe that Lemma \ref{lem:cvFonction} is straightforward if $\rho = 0$. Indeed, $P^{(n)}$ is an increasing function, and for any $t \in \R_+$, we have
\[
  \P(P^{(n)}_t = 0) = \left(1 - \P\left(\xi_1 > n^{1/2}\right)\right)^\floor{nt} \to 1 \quad \text{as }n \to \infty.
\]
Therefore, $\lim_{n \to \infty} P^{(n)} = 0$ in probability. Moreover, $\lim_{n \to \infty} S^{(n)} = \sigma B$ by Donsker's theorem. We conclude by Slutsky's theorem that $(S^{(n)},P^{(n)})$ converges toward $(\sigma B, 0)$ in law, for the Skorokhod topology.

We now assume that $\rho>0$. Let $t \in \R_+$, for any $\lambda,\mu \in \R$, we compute
\begin{align*}
  \E\left( \exp\left(i \lambda S^{(n)}_t + i \mu P^{(n)}_t\right) \right) &= \E\left( \exp\left(i  \tfrac{\lambda}{n^{1/2}}X_1 + i \mu \ind{\xi_1 > n^{1/2}} \right) \right)^\floor{nt} \\
  &= \left(\Phi(\lambda n^{-1/2}) + (e^{i\mu}-1)\P(\xi_1 > n^{1/2}) \Psi_n(\lambda)\right)^{\floor{nt}},
\end{align*}
where $\Phi(s) = \E(e^{i s X_1})$ and $\Psi_n(s) = \E( \exp(i s\tfrac{X_1}{n^{1/2}} )|\xi_1 > n^{1/2})$. Note that
\[
  \E\left( \left( \tfrac{X_1}{n^{1/2}} \right)^2 \middle| \xi_1 > n^{1/2} \right) = \frac{\E\left( X_1^2 \ind{\xi_1>n^{1/2}} \right)}{n\P(\xi_1 > n^{1/2})}.
\]
Thus by dominated convergence, we have
\[
  \lim_{n \to \infty} \E\left( \left( \tfrac{X_1}{n^{1/2}} \right)^2 \middle| \xi_1 > n^{1/2} \right) = \frac{1}{\rho}\lim_{n \to \infty}\E\left( X_1^2 \ind{\xi_1 > n^{1/2}} \right) = 0.
\]
Therefore, $\frac{X_1}{n^{1/2}}$ conditioned to $\{\xi_1 > n^{1/2}\}$ converges to 0 in $\mathrm{L}^2$, hence in law. This yields $\lim_{n \to \infty} \Psi_{n}(\lambda) = 1$ for all $\lambda \in \R$. Moreover, by \eqref{eqn:integrabilityrw} and \eqref{eqn:sequentialLimitXi} respectively, we have $\Phi(\lambda n^{-1/2}) - 1 \sim -\lambda^2\frac{\sigma^2}{2n}$ and $\P(\xi_1 > n^{1/2}) \sim \frac{\rho}{n}$ as $n \to \infty$. As a result, we have
\[
  \lim_{n \to \infty} \E\left( \exp(i \lambda S^{(n)}_t + i \mu P^{(n)}_t) \right) = \exp\left( - t\tfrac{\lambda^2\sigma^2}{2} - t \rho (e^{i\mu}-1) \right)
\]
which proves that $\lim_{n \to \infty} (S^{(n)}_t,P^{(n)}_t) = (\sigma B_t, P_t)$ for all $t>0$.

Using this result and the independent of the increments, we obtain the finite-dimensional distributions of $(S^{(n)},P^{(n)})$ toward $(\sigma B,P)$. By \cite[Theorem 5.1]{Sko57}, the finite-dimensional distribution convergence of a random walk in $\R^2$ implies the convergence in law for the Skorokhod topology of the process toward the associated Lévy process, concluding the proof.
\end{proof}

We use this convergence in law to prove the following result.
\begin{lemma}
\label{lem:ubrw}
Let $(a_n) \in \R_+^\N$ be a sequence such that $\lim_{n \to \infty} a_n = \infty$ and $a_n =o(n^{1/2})$. Under assumption \eqref{eqn:integrabilityrw}, if there exist an increasing sequence $(n_k) \in \N^\N$ and $\rho \in [0,\infty]$ such that $\lim_{k \to \infty} (a_{n_k})^2 \P(\xi_1 > a_{n_k}) = \rho$, then for any continuous functions $f<g$ such that $f(0)< 0 <g(0)$, for all $t >0$,
\begin{multline}
   \lim_{k \to \infty} \frac{a_{n_k}^2}{n_k} \log \P\left(\tfrac{S_j}{a_{n_k}} \in [f(j/n_k),g(j/n_k)], \tfrac{\xi_j}{a_{n_k}} \leq 1, j \in [tn_k] \right)\\
   = - \int_0^t  \frac{\pi^2 \sigma^2}{2(g(s)-f(s))^2} ds - \rho t.
   \label{eqn:ubrw}
\end{multline}
\end{lemma} 

\begin{proof}
Again, to simplify the notation, every asymptotic behaviour as $n \to \infty$ is implicitly understood as along the subsequence $(n_k)$.

First note that if $\rho = \infty$, then $\lim_{n \to \infty} a_{n}^2 \P(\xi_1 > a_{n}) = \infty$. As a result,
\begin{multline*}
  \limsup_{n \to \infty} \frac{a_n^2}{n} \log \P\left( \tfrac{S_j}{a_n} \in [f(j/n),g(j/n)], \tfrac{\xi_j}{a_n} \leq 1, j \in [tn] \right)\\
  \leq \limsup_{n \to \infty} \frac{a_n^2}{n} \log \P(\xi_1 \leq a_n)^\floor{tn} \leq \log \limsup_{n \to \infty} (1-\P(\xi_1 > a_n))^{ta_n^2} = -\infty,
\end{multline*}
concluding the proof in this case.

In the rest of the proof, we assume that $\rho<\infty$. We first prove that \eqref{eqn:ubrw} holds when functions $f$ and $g$ are constant. More precisely, given $t>0$,  $a< 0 <b$ and $c < d$ such that $a \leq c$ and $d \leq b$, and writing $I=[a,b]$, we first prove that
\begin{equation}
  \label{eqn:theAim}
  \liminf_{n \to \infty} \frac{a_n^2}{n} \log \P\left( \tfrac{S_\floor{tn}}{a_n} \in [c,d], \tfrac{S_j}{a_n} \in I, \xi_j \leq a_n, j \in [tn] \right) \geq - \tfrac{t\pi^2 \sigma^2}{2(b-a)^2} - t\rho.
\end{equation}
To prove this result, we will divide the time interval $[0,n]$ into $O(n/a_n^2)$ intervals of width $O(a_n^2)$. On each of these time intervals, the random walk trajectory, properly normalized, converges toward a Brownian motion. We finally use Brownian motion estimates to bound the probability for the random walk to leave the interval $a_n I$ in any of these time intervals.

Let $0<\epsilon < \frac{d-c}{8}$ and $T \in \N$. For any $n \in \N$, we set $r_n = \floor{T a_n^2}$ and denote by $K_n = \floor{\floor{tn}/r_n}-1$. For $k \in [K_n]$, we introduce the times $m_{n,k} = k r_n$ and set $m_{n,K_n+1} = \floor{tn}$. Applying the Markov property at times $m_{n,K_n}, m_{n,K_n - 1}, \ldots m_{n,1}$, and considering only trajectories that are at each time $m_{n,k}$ within distance $\epsilon a_n$ from $0$, we have
\begin{equation}
  \label{eqn:borneInf}
  \P\left( \tfrac{S_n}{a_n} \in [c,d], \tfrac{S_j}{a_n} \in I, \xi_j \leq a_n, j \in [t n] \right) \geq \pi_n^{K_n} \bar{\pi}_n,
\end{equation}
where we have set
\begin{align*}
  \pi_n &= \inf_{|h| \leq \epsilon a_n} \P_h\left( \left|\tfrac{S_{r_n}}{a_n}\right| \leq \epsilon, \tfrac{S_j}{a_n} \in I, \xi_j\leq a_n, j \in [r_n]\right)\\
  \bar{\pi}_n &= \inf_{|h| \leq \epsilon a_n} \P_h\left( \tfrac{S_{\Delta_n}}{a_n} \in [c,d], \tfrac{S_j}{a_n} \in I, \xi_j \leq a_n, j \in [\Delta_n] \right)
\end{align*}
and $\Delta_n = \floor{tn} - K_n r_n$.

We now study the asymptotic behaviour of $\pi_n$ as $n \to \infty$. We introduce the rescaled trajectories, defined for $s \in [0,T]$ by
\[
  S^{(n)}_s =a_n^{-1} S_{\floor{a_n^2 s}} \text{ and } P^{(n)}_s = \textstyle{\sum}_{j=1}^{\floor{a_n^2 s}} \ind{\xi_j > a_n}.
\]
We observe that the probability $\pi_n$ can be rewritten as
\begin{align*}
  \pi_n &= \inf_{|h| \leq \epsilon a_n} \P_h\left( P^{(n)}_T = 0, |S^{(n)}_T| \leq \epsilon, S^{(n)}_s \in I, s \leq T \right)\\
  &= \inf_{|h| \leq \epsilon} \P\left( P^{(n)}_T = 0, |S^{(n)}_{T}+h| \leq \epsilon, S^{(n)}_s + h \in I, s \leq T \right),
\end{align*}
using the fact that the law of $(S^{(n)},P^{(n)})$ under $\P_{ha_n}$ and $(S^{(n)}+h, P^{(n)})$ under~$\P_0$ are the same. Moreover, note that for all $h \in [0,\epsilon]$, we have
\begin{multline*}
  \P\left( P^{(n)}_T = 0, S^{(n)}_{T} \in [-\epsilon - h, \epsilon - h], S^{(n)}_s \in [a-h,b-h], s \leq T \right)\\
  \geq \P\left( P^{(n)}_T = 0, S^{(n)}_T \in [-\epsilon,0], S^{(n)}_s \in [a,b-\epsilon], s \leq T \right),
\end{multline*}
as $[-\epsilon,0] \subset [-\epsilon - h,\epsilon - h]$ and $[a,b-\epsilon] \subset [a-h,b-h]$ for all $h \in [0,\epsilon]$. Similarly, for all $h \in [0,-\epsilon]$, we have
\begin{multline*}
  \P\left( P^{(n)}_T = 0, S^{(n)}_{T} \in [-\epsilon - h, \epsilon - h], S^{(n)}_s \in [a-h,b-h], s \leq T \right)\\
  \geq \P\left( P^{(n)}_T = 0, S^{(n)}_T \in [0,\epsilon], S^{(n)}_s \in [a+\epsilon,b], s \leq T \right).
\end{multline*}
Rewriting the infimum over the interval $[-\epsilon, \epsilon]$ as the minimum of the infimum over the intervals $[-\epsilon,0]$ and $[0,\epsilon]$, we obtain
\[
  \pi_n \geq 
  \min_{\delta \in \{-\epsilon,0\}} \P \left( P^{(n)}_T = 0, S^{(n)}_T+\delta \in (-\epsilon,0), S^{(n)}_s+\delta \in (a, b - \epsilon), s \in [0,T] \right).
\]

By Lemma \ref{lem:cvFonction}, we know that $(S^{(n)},P^{(n)})$ converges toward $(\sigma B,P)$, where $B$ is a Brownian motion and $P$ an independent Poisson process with intensity $\rho$. Therefore, by Portmanteau theorem, we have
\[
  \liminf_{n \to \infty} \pi_n \geq  \min_{\delta \in \{-\epsilon,0\}}  \P\left( \sigma B_T + \delta \in (-\epsilon, 0), P_T = 0, \sigma B_s \in (a,b-\epsilon), s \in [0,T] \right).
\]
Using similar estimates, we obtain that $\liminf_{n \to \infty} \bar{\pi}_n > 0$.

Note that $K_n \sim \frac{tn}{Ta_n^2}$ as $n \to \infty$. Therefore, \eqref{eqn:borneInf} yields
\begin{multline*}
  \liminf_{n \to \infty} \frac{a_n^2}{n} \log \P\left( \tfrac{S_n}{a_n} \in [c,d], \tfrac{S_j}{a_n} \in I, \xi_j \leq a_n, j \in [t n] \right)\\
  \geq \liminf_{n \to \infty} \frac{a_n^2}{n} \left( K_n \log \pi_n + \log \tilde{\pi}_n \right) \geq \frac{t}{T} \liminf_{n \to \infty} \log \pi_n.
\end{multline*}
Thus, using the above estimates and the independence between $P$ and $B$, we obtain that for all $T > 0$,
\begin{align}
  &\liminf_{n \to \infty} \frac{a_n^2}{n} \log \P\left( \tfrac{S_n}{a_n} \in [c,d], \tfrac{S_j}{a_n} \in I, \xi_j \leq a_n, j \in [t n] \right) \nonumber\\
  \geq &- \rho t +  \frac{t}{T} \log \min_{\delta \in \{-\epsilon,0\}}\P\left( \sigma B_T + \delta \in (-\epsilon,0), \sigma B_s + \delta \in (a, b-\epsilon), s \in [0,T] \right). \label{eqn:temp}
\end{align}
Using \cite[Chapter 1.7, Problem 8]{IMK74}, we know that for all $|\delta| \leq \epsilon$, we have
\begin{multline*}
  \lim_{T \to \infty} \frac{1}{T} \log \P\left(\sigma B_T + \delta \in (-\epsilon,0), \sigma B_s + \delta \in (a, b-\epsilon), s \in [0,T] \right)\\ = - \frac{\pi^2 \sigma^2}{2(b-a - \epsilon)^2}.
\end{multline*}
Therefore, letting $T \to \infty$ in \eqref{eqn:temp}, then $\epsilon \to 0$ yields \eqref{eqn:theAim}.

In a second time, we prove a uniform version of \eqref{eqn:theAim}. Let $a<x<x'<b$, we choose $0<\epsilon<\frac{\min\{(d-c),(b-x')\}}{8})$ and set $N = \ceil{\frac{(x'-x)}{\epsilon}}$. We note that
\begin{align*}
  &\inf_{h \in [x,x']} \P_{ha_n}\left( \tfrac{S_n}{a_n} \in [c,d], \tfrac{S_j}{a_n} \in I, \xi_j \leq a_n, j \in [tn] \right)\\
  \geq & \min_{j \in [N]} \inf_{h \in [x+j\epsilon,x+(j+1)\epsilon]} \P\left( \tfrac{S_n}{a_n} + h \in [c,d], \tfrac{S_j}{a_n} + h \in I, \xi_j \leq a_n, j \in [tn] \right)\\
  \geq & \min_{j \in [N]} \P_{(x + j\epsilon)a_n}\left( \tfrac{S_n}{a_n} \in [c,d-\epsilon], \tfrac{S_j}{a_n} \in [a,b-\epsilon], \xi_j \leq a_n, j \in [tn] \right),
\end{align*}
using the same techniques as the ones used to bound $\pi_n$. Thus, applying \eqref{eqn:theAim} to let $n \to \infty$, then letting $\epsilon \to 0$ we obtain
\begin{multline}
  \label{eqn:cvUnif}
  \liminf_{n \to \infty} \frac{a_n^2}{n} \inf_{h \in [x,x']} \log \P_{ha_n}\left( \tfrac{S_n}{a_n} \in [c,d], \tfrac{S_j}{a_n} \in [a,b], \xi_j \leq a_n, j \in [tn] \right)\\ \geq - \frac{t\pi^2 \sigma^2}{2(b-a)^2} - t\rho.
\end{multline}

Finally, using \eqref{eqn:cvUnif}, we can tackle general continuous functions. Let $f<g$ be two continuous functions such that $f(0)<0<g(0)$. We introduce a continuous function $h$ such that $f<h<g$ and $h(0) = 0$. Let $K \in \N$, for $k \in [K]$ we set $I_k = [\tfrac{k-1}{K},\tfrac{k+2}{K}]\cap [0,1]$, as well as
\[
  f_k = \sup_{s \in I_k} f(s), \quad g_k = \inf_{s \in I_k} g(s) \quad \text{and} \quad h_k = h(k/K).
\]
We choose some large $K \in \N$ and some small $\epsilon > 0$. Applying the Markov property at times $k\floor{n/K}$ for $k \in [K]$, we obtain, for all $n \in \N$ large enough
\begin{align*}
  &\P\left(\tfrac{S_j}{a_{n}} \in [f(j/n),g(j/n)], \xi_j \leq a_n , j \in [tn] \right)\\
  \geq & \prod_{k=0}^{\floor{tK}} \inf_{|s-h_k| \leq \epsilon} \P_{sa_n} \left( \left|\tfrac{S_\floor{n/K}}{a_n}-h_{k+1}\right| \leq \epsilon, \tfrac{S_j}{a_n} \in [f_k,g_k], \xi_j \leq a_n, j \in [\floor{n/K}] \right).
\end{align*}
Therefore, using \eqref{eqn:cvUnif}, we obtain
\begin{multline*}
  \liminf_{n \to \infty} \frac{a_n^2}{n} \log \P\left(\tfrac{S_j}{a_{n}} \in [f(j/n),g(j/n)], \xi_j \leq a_n , j \in [tn] \right) \\
  \geq - \sum_{k=0}^\floor{tK} \frac{\pi^2 \sigma^2}{2K(g_k^2 - f_k^2)^2} - \rho \frac{\floor{tK} + 1}{K}.
\end{multline*}
Letting $K \to \infty$, we obtain the lower bound of \eqref{eqn:ubrw}.

The upper bound for this equation is obtained in a very similar way, replacing $\inf$ by $\sup$, $\floor{\cdot}$ by $\ceil{\cdot}$ and $\min$ by $\max$. More precisely, one can prove an analogue of~\eqref{eqn:cvUnif}, by dividing the time interval $[0,n]$ into $O(n/a_n^2)$ intervals of length $O(a_n^2)$, using the Markov property and replacing $\inf_{|h| \leq \epsilon a_n}$ by $\sup_{h \in [aa_n,ba_n]}$. Then, approximating functions $f$ and $g$ by staircase functions and using again the Markov property, the analogue of \eqref{eqn:cvUnif} yields the upper bound of \eqref{eqn:ubrw}.
\end{proof}

Finally, to prove Theorem \ref{thm:newSpinalMogulskii}, we observe that under its assumptions, for any $\lambda > 0$ we have $\lim_{n \to \infty} a_n^2 \P(\xi_1 > \lambda a_n) = \tfrac{\rho}{\lambda^2}$. Using Lemma \ref{lem:ubrw} with the increasing sequence $(n)$, the proof is concluded.

\section{Tail of the consistent maximal displacement}
\label{sec:computation}

We denote by $\emptyset$ the ancestor of the branching random walk. For any $u \in \T$, we write $\pi u$ for the parent of $u$, $\Omega(u)$ for the set of children of~$u$,
\[
  \tilde{\xi}(u) =  \log \left(\sum_{v \in \Omega(u)} e^{V(u)-V(v)}\right) \quad \text{and} \quad \xi(u) = \begin{cases} \tilde{\xi}(\pi u) & \mathrm{if} \quad u \neq \emptyset\\ 0 &\mathrm{if} \quad u = \emptyset.\end{cases}
\]
Note that by \eqref{eqn:boundary}, we have $\P(\tilde{\xi}(\emptyset) > x) \leq e^{-x}$ for any $x \in \R_+$. We set
\begin{equation}
\label{eqn:lambdaDefinition}
  \lambda^* = \left(\frac{ 3\pi^2 \sigma^2 }{2} \right)^{1/3}.
\end{equation}

In this section, we obtain upper and lower bounds for the left tail of $L_n$, ultimately obtaining the following large deviation estimate for the consistent maximal displacement.
\begin{theorem}
\label{thm:largedeviations}
Under assumptions \eqref{eqn:supercritical}, \eqref{eqn:boundary}, \eqref{eqn:variance} and \eqref{eqn:integrability}, we have for all $\lambda \in [0,\lambda^*]$
\begin{equation}
  \label{eqn:largedeviations}
  \lim_{n \to \infty} \frac{1}{n^{1/3}} \log \P(L_n \leq \lambda n^{1/3}) = \lambda- \lambda^*.
\end{equation}
\end{theorem}

In a first time, we prove the upper bound in Theorem \ref{thm:largedeviations}, for which the assumption \eqref{eqn:integrability} is not necessary.
\begin{lemma}
\label{lem:lowerbound}
We assume \eqref{eqn:supercritical}, \eqref{eqn:boundary} and \eqref{eqn:variance}. For any $\lambda \in (0,\lambda^*)$, we have
\[\limsup_{n \to \infty} \frac{1}{n^{1/3}} \log \P\left( L_n \leq \lambda n^{1/3}\right) \leq \lambda - \lambda^*.\]
\end{lemma}

\begin{proof}
Let $\lambda \in (0,\lambda^*)$, we set $f : t \mapsto \lambda -\lambda^*(1-t)^{1/3}$. For $n \in \N$, write
\[
  Y_n = \sum_{|u| \leq n} \ind{V(u) < f(|u|/n)n^{1/3}}\ind{V(u_j) \in [f(j/n)n^{1/3},\lambda n^{1/3}], j \in [|u|-1]}.
\]
As $f(0)<0$ and $f(1) \geq \lambda$, we have by Markov inequality
\begin{equation}
  \label{eqn:dmb}\P(L_n \leq \lambda n^{1/3}) \leq \P(Y_n \geq 1) \leq \E(Y_n).
\end{equation}

Using \cite{Mal15a}, we can compute the asymptotic behaviour of $\E(Y_n)$. More precisely, in that article, the branching random walks that are considered are such that $(\T,-V)$ is in the boundary case. Therefore, \cite[Lemma 3.1]{Mal15a} yields
\begin{align*}
  \limsup_{n \to \infty} n^{-1/3} \log \E(Y_n) &\leq \sup_{t \in [0,1]}\left( f(t) - \int_0^t \frac{\pi^2 \sigma^2}{2 (\lambda - f(s))^2} ds \right)\\
  &\leq \sup_{t \in [0,1]} \left(f(t) - \frac{\pi^2 \sigma^2}{2(\lambda^*)^2} 3\left(1 - (1-t)^{1/3}\right) \right)\\
  &\leq \sup_{t \in [0,1]} \left(f(t) - \lambda^* \left(1 - (1 - t)^{1/3}\right) \right) = \lambda - \lambda^*,
\end{align*}
as $\lambda^* = \frac{3 \pi^2 \sigma^2}{2(\lambda^*)^2}$ by definition \eqref{eqn:lambdaDefinition}.

As a result, we conclude from \eqref{eqn:dmb} that
\[
  \limsup_{n \to \infty} n^{-1/3} \log \E(Y_n) \leq \lambda - \lambda^*.\qedhere
\]
\end{proof}

We now assume that \eqref{eqn:integrability} does not hold, and use Lemma \ref{lem:ubrw} instead of the usual Mogul'skii estimate to improve the upper bound in the asymptotic left tail of $L_n$ along a well-chosen subsequence. In particular, this shows that \eqref{eqn:integrability} is necessary for Theorem \ref{thm:largedeviations} to hold.
\begin{lemma}
\label{lem:ubliminf}
We assume \eqref{eqn:supercritical}, \eqref{eqn:boundary}, \eqref{eqn:variance} and $\limsup_{x \to \infty} x^2 \hat{\P}(\xi_1 > x) > 0$. There exists $\lambda>\lambda^*$ such that 
\[\liminf_{n \to \infty} \frac{1}{n^{1/3}} \log \P\left( L_n \leq \lambda n^{1/3}\right) < 0.\]
\end{lemma}

\begin{proof}
Let $\rho \in (0,\infty]$ and $(x_k) \in \R^\N$ be such that $\lim_{n \to \infty} x_n = \infty$ and
$\lim_{k \to \infty} x_k^2 \hat{\P}(\xi_1 > x_k) = \rho$. We set $R = 3\lambda^*$ and $n_k = \floor{(x_k/R)^3}$. We note that $\liminf_{k \to \infty} (n_k)^{2/3}\hat{\P}(\xi_1 > R(n_k)^{1/3}) \geq \frac{\rho}{R^2}$. Up to modifying $\rho$ and extracting a subsequence from $(n_k)$, we may assume without loss of generality that
\begin{equation}
  \label{eqn:bonneSousSuite}
  \lim_{k \to \infty} (n_k)^{2/3}\hat{\P}(\xi_1 > R (n_k)^{1/3}) = \rho.
\end{equation}

Let $\lambda \in (0,2\lambda^*)$ and $f$ be a continuous increasing function that satisfies $f(0)\in (-\lambda^*,0)$ and $f(1) = \lambda$. For $k \in [n]$, we set $I^{(n)}_k = \left[ f(k/n)n^{1/3}, \lambda n^{1/3} \right]$ and we denote by
\[
  G_n = \left\{ u \in \T : |u| \leq n, V(u_j) \in I^{(n)}_j, \xi(u_j) \leq R n^{1/3}, j \in [|u|] \right\}.
\]
For $k\in [n]$ with $k>0$, we introduce the quantities
\[  X^{(n)}_k = \sum_{|u| = k} \ind{\pi u \in G_n, V(u) < f(k/n)n^{1/3}} \text{ and }  Y^{(n)}_k = \sum_{|u| = k-1} \ind{u \in G_n,\tilde{\xi}(u) > R n^{1/3}}.
\]
We observe that
\begin{align}
  \P\left( L_n \leq \lambda n^{1/3}\right)
  =  &\P\left( \exists |u| = n : \forall j \in [n], V(u_j) \leq \lambda n^{1/3} \right)\nonumber\\
  \leq &\P\Bigg( \sum_{j=1}^n X^{(n)}_j + Y^{(n)}_j \geq 1 \Bigg)
  \leq \sum_{j=1}^n \E\left( X^{(n)}_j + Y^{(n)}_j \right) . \label{eqn:markov}
\end{align}

Using the spinal decomposition we have
\begin{align*}
  \E\left( X^{(n)}_k \right)
  = &\hat{\E} \Bigg( \sum_{|u|=k} \frac{e^{-V(u)}}{W_k} e^{V(u)}\ind{V(u) < f(k/n)n^{1/3}} \ind{\pi u \in G_n}  \Bigg)\\
  = &\hat{\E} \left( e^{V(w_k)}\ind{V(w_k) < f(k/n)n^{1/3}} \ind{w_{k-1} \in G_n} \right)\\
  \leq &e^{f(k/n)n^{1/3}} \hat{\P}\left( w_{k-1} \in G_n \right).
\end{align*}
Similarly, as for any $|u|=k-1$, $\tilde{\xi}(u)$ is independent of $\calF_{k-1}$, and by \eqref{eqn:boundary}, we have $\E(e^{\tilde{\xi}(u)}) = 1$ for all $u \in \T$. Therefore
\begin{align*}
  \E\left( Y^{(n)}_k \right) &= \E\Bigg( \sum_{|u|=k-1} \ind{u \in G_n} \Bigg) \P\left( \tilde{\xi}(u) > R n^{1/3} \right)\\
  &\leq e^{-R n^{1/3}} \hat{\E} \left( e^{V(w_{k-1})}  \ind{w_{k-1} \in G_n} \right)
  \leq e^{(\lambda - R) n^{1/3}} \hat{\P}\left( w_{k-1} \in G_n \right).
\end{align*}
Consequently, as $\lambda - R < -\lambda^* < f(t)$ for any $t \in [0,1]$, \eqref{eqn:markov} becomes
\begin{equation*}
   \P\left( L_n \leq \lambda n^{1/3} \right) \leq 2\sum_{k=1}^n  e^{f(k/n)n^{1/3}} \hat{\P}\left( w_{k-1} \in G_n \right).
\end{equation*}

Let $A \in \N$, for any $a \in [A]$, we write $m_a = \floor{na/A}$. As $f$ is increasing, for any $a \in [A-1]$ and $k \in (m_a,m_{a+1}] \cap \N$, we have
\[
  e^{ f(k/n)n^{1/3}} \hat{\P}\left( w_{k-1} \in G_n \right) \leq e^{ f((a+1)/A)n^{1/3}} \hat{\P}\left( w_{m_a} \in G_n \right).
\]
Moreover, by the spinal decomposition, $(V(w_j),\xi(w_j), j \in \Z_+)$ is an enriched random walk under law $\hat{\P}$. By \eqref{eqn:bonneSousSuite}, we use Lemma \ref{lem:ubrw} to obtain for any $a \in [A-1]$ with $a \neq 0$,
\[
  \lim_{k \to \infty} \frac{1}{(n_k)^{1/3}} \log \hat{\P} \left( w_{m_a} \in G_{n_k} \right) = - \int_0^{a/A} \frac{\pi^2 \sigma^2}{2(\lambda - f(s))^2} ds - \frac{\rho a}{A}.
\]
This limit also trivially holds for $a=0$ (with the convention $0.\infty=0$). Therefore, letting $n \to \infty$ along the subsequence $(n_k)$, \eqref{eqn:markov} yields
\begin{multline*}
  \liminf_{n \to \infty} \frac{1}{n^{1/3}} \log \P\left( L_n \leq \lambda n^{1/3} \right)\\
   \leq \liminf_{n \to \infty} \frac{1}{n^{1/3}} \log \left(\ceil{\frac{n}{A}} \sum_{a=0}^{A-1} e^{ f((a+1)/A)n^{1/3}} \hat{\P}\left( w_{m_a} \in G_n \right) \right)\\
   \leq \max_{a \in [A]} \left( f((a+1)/A) - \int_0^{a/A} \frac{\pi^2 \sigma^2}{2(\lambda - f(s))^2} ds - \frac{\rho a}{A} \right).
\end{multline*}
Then, letting $A \to \infty$, we obtain
\begin{equation}
  \label{eqn:unelim}
  \liminf_{n \to \infty} \frac{1}{n^{1/3}} \log \P\left( L_n \leq \lambda n^{1/3} \right) \leq \sup_{t \in [0,1]}\left( f(t) - \int_0^t \frac{\pi^2 \sigma^2}{2(\lambda - f(s))^2} ds - \rho t\right).
\end{equation}

Note that if $\rho=\infty$, we can choose $f : t \mapsto - \lambda^*/2 + 2 \lambda^* t$ and $\lambda = 3 \lambda^*/2$. In that case, \eqref{eqn:unelim} allows to conclude the proof.

Thus, in the rest of the proof, we assume that $\rho<\infty$. For all $\lambda \geq \lambda^*$ and $\mu \geq 0$, we denote by $f_{\lambda,\mu}$ the solution of
\[
  \forall t \in [0,1), \  y'(t) = \tfrac{\pi^2 \sigma^2}{2}(\lambda - y(t))^{-2} + \mu,  \quad \text{with } y(1) = \lambda
\]
Note that $f_{\lambda,\mu}$ is continuous with respect to $(\lambda,\mu)$, decreasing in $\mu$, and that
\[
  \forall t \in [0,1], f_{\lambda,0}(t) =\lambda - \lambda^*(1-t)^{1/3}
\]
Thus for a given $\rho>0$, there exists $\lambda > \lambda^*$ close enough to $\lambda^*$ such that  $f_{\lambda,\rho}(0)<0$. Applying \eqref{eqn:unelim} with this choice of $\lambda$ and $f=f_{\lambda,\rho}$ yields
\[
  \liminf_{n \to \infty} \frac{1}{n^{1/3}} \log \P\left( L_n \leq \lambda n^{1/3} \right) \leq f_{\lambda,\rho}(0)<0,
\]
which concludes the proof.
\end{proof}

The two previous lemmas allow to bound from below the consistent maximal displacement, showing in particular that with high probability $L_n \geq (\lambda^*-\epsilon)n^{1/3}$ for all $\epsilon>0$ and $n$ large enough. To bound $L_n$ from above, we prove that with high probability there exists an individual staying below $\lambda n^{1/3}$ for $n$ units of time, as soon as $\lambda$ is large enough, using a second moment computation.
\begin{lemma}
\label{lem:upperbound}
We assume \eqref{eqn:supercritical}, \eqref{eqn:boundary}, \eqref{eqn:variance} and \eqref{eqn:integrability}. For any $\lambda \in (0, \lambda^*)$, we have
\[\liminf_{n \to \infty} \frac{1}{n^{1/3}} \log \P\left( L_n \leq \lambda n^{1/3} \right) \geq \lambda - \lambda^*.\]
\end{lemma}

\begin{proof}
Let $\lambda \in (0,\lambda^*)$, $\delta > 0$, and $f : t \in [0,1] \mapsto \lambda - \lambda^*(1+\delta-t)^{1/3}$. We denote by $I^{(n)}_j = \left[ f(j/n)n^{1/3}, \lambda n^{1/3} \right]$ for all $j \in [n]$. We set
\[
  Z_n = \sum_{|u|=n} \ind{V(u_j) \in I^{(n)}_j, \xi(u_j) \leq \delta n^{1/3}, j \in [n]}.
\]
We compute the first two moments of $Z_n$ to bound from below $\P(Z_n>0)$.

Using the spinal decomposition, we have
\begin{align*}
  \E(Z_n)
  &= \hat{\E}\left( e^{V(w_n)} \ind{V(w_j) \in I^{(n)}_j, \xi(w_{j}) \leq \delta n^{1/3}, j \in [n]}\right)\\
  &\geq e^{f(1)n^{1/3}} \hat{\P}\left( V(w_j) \in I^{(n)}_j, \xi(w_{j}) \leq \delta n^{1/3}, j \in [n] \right).
\end{align*}
As $\displaystyle \lim_{n \to \infty} n^{2/3} \hat{\P}( \xi(w_1) > \delta n^{1/3} ) = 0$ by \eqref{eqn:integrability}, Theorem \ref{thm:newSpinalMogulskii} yields
\begin{equation}
  \label{eqn:firstmoment}
  \liminf_{n \to \infty} \frac{1}{n^{1/3}} \log \E(Z_n) \geq f(1) - \frac{\pi^2 \sigma^2}{2}\int_0^1 \frac{ds}{(\lambda-f(s))^2}
  \geq \lambda  - \lambda^*(1 + \delta)^{1/3}. 
\end{equation}

Similarly, to compute the second moment we observe that
\begin{align*}
  \E(Z_n^2) &= \hat{\E}\Bigg( Z_n\sum_{|u|=n} \frac{e^{-V(u)}}{W_n} e^{V(u)} \ind{V(u_j) \in I^{(n)}_j,\xi(u_{j}) \leq \delta n^{1/3}, j \in [n]} \Bigg)\\
  &= \hat{\E}\left(Z_n e^{V(w_n)} \ind{V(w_j) \in I^{(n)}_j, \xi(w_{j}) \leq \delta n^{1/3}, j \in [n]}\right)\\
  &\leq e^{\lambda n^{1/3}} \hat{\E}\left(Z_n \ind{V(w_j) \in I^{(n)}_j, \xi(w_{j}) \leq \delta n^{1/3}, j \in [n]}\right).
\end{align*}
Under the law $\hat{\P}$, $Z_n$ can be decomposed as follows
\[
  Z_n = \ind{V(w_j) \in I^{(n)}_j, \xi(w_{j}) \leq \delta n^{1/3}, j \in [n]} + \sum_{k=0}^{n-1} \sum_{u \in \Omega(w_k), u \neq w_{k+1}} Z_n(u),
\]
where $Z_n(u) = \sum_{|v| =n, v > u} \ind{V(u_j) \in I^{(n)}_j,\xi(u_{j}) \leq \delta n^{1/3}, j \in [n]}$. We denote by
\[
  \calG = \sigma\left( w_n, \Omega(w_n), V(u), u \in \Omega(w_n), n \in \N \right).
\]

Observe that conditionally on $\calG$, for any $u \in \Omega(w_k)$ such that $u \neq w_{k+1}$, the subtree of the descendants of $u$ has the law of a branching random walk starting from $V(u)$. Therefore, writing $\P_x$ for the law of $(\T,V+x)$, for any $k\in [n-1]$ and $u \in \Omega(w_k)$ such that $u \neq w_{k+1}$, we have
\begin{align*}
  \hat{\E}\left( Z_n(u) | \calG \right) &\leq \E_{V(u)}\Bigg( \sum_{|v| = n-k-1} \ind{V(v_j) \in I^{(n)}_{k+j+1},\xi(v_{j}) \leq \delta n^{1/3}, j \in [n-k-1]} \Bigg)\\
  &\leq \E_{V(u)} \Bigg(\sum_{|v| = n-k-1} \ind{V(v_j) \in I^{(n)}_{k+j+1}, j \in [n-k-1]} \Bigg).
\end{align*}
Applying spinal decomposition, for any $x \in \R$ and $p \in [n]$,
\begin{multline*}
  \E_{x} \Bigg(\sum_{|v| = n-p} \ind{V(v_j) \in I^{(n)}_{p+j}, j \in [n-p]} \Bigg) = \hat{\E}\Bigg( e^{V(w_{n-p})} \ind{V(w_j)+x \in I^{(n)}_{p+j}, j \in [n-p]} \Bigg)\\
  \leq e^{\lambda n^{1/3}-x} \hat{\P}\left( V(w_j) + x \in I^{(n)}_{p+j}, j \in [n-p] \right).
\end{multline*}
Let $A \in \N$, for any $a \in [A]$ we set $m_a = \floor{na/A}$ and
\[
  \Psi^{(n)}_{a,A} = \sup_{y \in I^{(n)}_{m_{a}}} \hat{\P}\left( V(w_j) + y \in I^{(n)}_{m_{a}+j}, j \in [n-m_{a}] \right).
\]
Using the previous equation, for any $a \in [A-1]$ and $k \in [n-1]$ such that $m_a \leq k < m_{a+1}$, we have
\begin{align*}
  \sum_{\substack{u \in \Omega(w_k)\\ u \neq w_{k+1}}} \hat{\E}\left( Z_n(u) | \calG \right)
  &\leq e^{\lambda n^{1/3}} \Psi^{(n)}_{a+1,A} \sum_{u \in \Omega(w_k)} e^{-V(u)}\\
  &\leq e^{\lambda n^{1/3} - V(w_k) + \xi(w_{k+1})} \Psi^{(n)}_{a+1,A}.
\end{align*}
As $V(w_{k}) \geq f(m_a/n) n^{1/3}$ (using the fact that $f$ is increasing), we obtain
\begin{multline*}
  \E(Z_n^2)
  \leq e^{\lambda n^{1/3}} \hat{\P}\left( V(w_j) \in I^{(n)}_j, j \in [n] \right)\\
  + e^{(2 \lambda + \delta) n^{1/3}} \sum_{a=0}^{A-1} n \Psi^{(n)}_{a+1,A} e^{- f(m_a/n)n^{1/3}} \hat{\P}\left( V(w_j) \in I^{(n)}_j, j \in [n] \right).
\end{multline*}
Therefore, applying Theorem \ref{thm:newSpinalMogulskii} (and the fact that $\Psi^{(n)}_{A,A}=1$), we have
\begin{multline*}
  \limsup_{n \to \infty} \frac{1}{n^{1/3}} \log \E(Z_n^2) \leq 2 \lambda + \delta - \frac{\pi^2 \sigma^2}{2} \int_0^1 \frac{ds}{(\lambda - f(s))^2}\\
  + \max_{a \in [A-1]} \left(- f(a/A) - \frac{\pi^2 \sigma^2}{2} \int_{(a+1)/A}^1 \frac{ds}{(\lambda - f(s))^2} \right).
\end{multline*}
Letting $A \to \infty$, we obtain
\begin{equation}
  \label{eqn:secondmoment}
  \limsup_{n \to \infty} \frac{1}{n^{1/3}} \log \E(Z_n^2) \leq \lambda - \lambda^*(1 + \delta)^{1/3} +\delta + 2 \lambda^*\delta^{1/3}.
\end{equation}

By Cauchy-Schwarz inequality, we have
\[\P\left( L_n \leq \lambda n^{1/3} \right) \geq \P\left( Z_n > 0 \right) \geq \frac{\E(Z_n)^2}{\E(Z_n^2)}.\]
Thus, using \eqref{eqn:firstmoment} and \eqref{eqn:secondmoment}, we obtain
\[
  \liminf_{n \to \infty} \frac{1}{n^{1/3}} \log \P\left( L_n \leq \lambda n^{1/3} \right) \geq \lambda - \lambda^*(1 + \delta)^{1/3} - \delta - 2 \lambda^*\delta^{1/3}.
\]
Letting $\delta \to 0$ allows to conclude.
\end{proof}

\begin{proof}[Proof of Theorem \ref{thm:largedeviations}]
Using Lemmas \ref{lem:lowerbound} and \ref{lem:lowerbound}, we observe that assuming that \eqref{eqn:integrability}, we have
\[
  \lim_{n \to \infty} \frac{1}{n^{1/3}} \log \P(L_n \leq \lambda n^{1/3}) = \lambda - \lambda^*,
\]
which concludes the proof.
\end{proof}

\begin{proof}[Proof of Theorem \ref{thm:main}]
We now extend Theorem \ref{thm:largedeviations} to obtain the almost sure asymptotic behaviour of $L_n$. We first assume that \eqref{eqn:integrability} does not hold. By Lemma~\ref{lem:ubliminf}, there exists  $\lambda > \lambda^*$ and an increasing sequence $(n_k)\in \N^\N$ such that
\[\sum_{k =1}^{\infty} \P( L_{n_k} \leq n_k^{1/3}\lambda) < \infty.\]
Therefore, by Borel-Cantelli lemma, we have $\liminf_{k \to \infty} \frac{L_{n_k}}{{n_k}^{1/3}} \geq \lambda$ a.s. Therefore, we conclude that
\[
  \limsup_{n \to \infty} \frac{L_n}{n^{1/3}} > \lambda^* \quad \text{a.s,}
\]
proving that \eqref{eqn:integrability} is necessary for the convergence in Theorem~\ref{thm:main} to hold.

We assume in the rest of the proof that \eqref{eqn:integrability} holds. By Lemma~\ref{lem:lowerbound}, for any $\lambda < \lambda^*$, we have $\sum_{n =1}^{\infty} \P( L_{n} \leq n^{1/3}\lambda) < \infty$. As a result, we have
\[
  \liminf_{n \to \infty} \frac{L_n}{n^{1/3}} \geq \lambda^* \text{ a.s.}
\]
by taking $\lambda \to \lambda^*$.

We now bound $L_n$ from above. By Lemma \ref{lem:upperbound}, for any $\delta > 0$, we have
\[
  \liminf_{n \to \infty} \frac{1}{n^{1/3}} \log \P(L_n \leq \lambda^* n^{1/3}) > -\delta.
\]
We work in the rest of the proof conditionally on the survival event $S$. We write $\hat{\T}$ for the subtree of $\T$ consisting of individuals having an infinite line of descent. By \cite[Chapter 1, Theorem 12.1]{AtNbook}, $\hat{\T}$ is a supercritical Galton-Watson process that never dies out. Applying \cite[Lemma 2.4]{Mal15a}\footnote{Note the integrability hypothesis \eqref{eqn:integrability} is weaker than \cite[Assumption (1.3)]{Mal15a}, but is enough to make the proof of \cite[Lemma 2.4]{Mal15a} hold.} to the branching random walk $(\hat{\T},V)$, there exists $a>0$ and $\rho>1$ such that the event
\[
  \mathcal{A}(p) = \left\{ \# \left\{ |u| = p : \forall j \leq p, V(u_j) \leq pa  \right\} \geq \rho^p \right\}
\]
is verified a.s. for $p\in \N$ large enough. Let $\eta>0$, we set $p=\floor{\eta n^{1/3}}$. Applying the Markov property at time $p$, we have
\[
  \P\left( L_{n+p} \geq (\lambda^* + a \eta)n^{1/3} | \mathcal{A}(p) \right) \leq \left( 1 - \P\left( L_{n} \leq \lambda^* n^{1/3} \right) \right)^{\rho^p}.
\]
Using the Borel-Cantelli lemma, we conclude that $\limsup_{n \to \infty} \frac{L_n}{n^{1/3}} \leq \lambda^* + a\eta$ a.s. on $S$.
We let $\eta \to 0$ to conclude the proof.
\end{proof}

\bibliographystyle{plain}

\def\cprime{$'$}

\end{document}